\documentclass[12pt]{amsart}
\usepackage[english]{babel}
\usepackage{amssymb, amsmath, amsthm}
\usepackage[utf8]{inputenc} 
\usepackage[T1]{fontenc}
\usepackage[shortlabels]{enumitem}
\usepackage{hyperref}
\usepackage[a4paper,top=2cm,bottom=2cm,left=2cm,right=2cm,marginparwidth=1.75cm]{geometry}

\usepackage{graphicx}
\graphicspath{ {./images/} }
\usepackage{subcaption}
\usepackage{float}

\newlist{symbols}{itemize}{1}
\setlist[symbols,1]{label=,labelwidth=1.5in,align=parleft,itemsep=0.1\baselineskip,leftmargin=!}

\usepackage{epsfig}
\usepackage{url}
\usepackage{setspace}
\usepackage{color}
\usepackage{mathrsfs}
\usepackage{amsmath}
\usepackage{amsthm}
\usepackage{stmaryrd}

\usepackage{tikz}
\usepackage{pgf}
\usetikzlibrary{arrows}
\usetikzlibrary{automata,positioning}

\newtheorem{theorem}{Theorem}[section]
\newtheorem{cor}[theorem]{Corollary}
\newtheorem{lem}[theorem]{Lemma}

\newtheorem{conj}[theorem]{Conjecture}

\newtheorem{exam}[theorem]{Example}

\theoremstyle{definition}

\theoremstyle{remark}
\newtheorem{rem}[theorem]{Remark}

\renewcommand{\phi}{\varphi}

\newcommand{\lcm}{{\rm lcm}}

\allowdisplaybreaks

\usepackage{amsaddr}

\title{High order congruences for $M$-ary partitions}

\author{B\l{}a\.zej \.Zmija}
\address{Charles University, Faculty of Mathematics and Physics, Department of Algebra, Sokolov\-sk\' a 83, 18600 Praha~8, Czech Republic} 
\address{Jagiellonian University in Krak\'ow, Faculty of Mathematics and Computer Science, Institute of Mathematics, ul. {\L}ojasiewicza 6, 30-348 Krak\'ow, Poland}
\email{blazej.zmija@matfyz.cuni.cz}
\date{\today}

\begin{document}

\begin{abstract}
For a sequence $M=(m_{i})_{i=0}^{\infty}$ of integers such that $m_{0}=1$, $m_{i}\geq 2$ for $i\geq 1$, let $p_{M}(n)$ denote the number of partitions of $n$ into parts of the form $m_{0}m_{1}\cdots m_{r}$. In this paper we show that for every positive integer $n$ the following congruence is true:
\begin{align*}
p_{M}(m_{1}m_{2}\cdots m_{r}n-1)\equiv 0\ \ \left({\rm mod}\ \prod_{t=2}^{r}\mathcal{M}(m_{t},t-1)\right),
\end{align*} 
where $\mathcal{M}(m,r):=\frac{m}{\gcd\big(m,\lcm (1,\ldots ,r)\big)}$.

Our result answers a conjecture posed by Folsom, Homma, Ryu and Tong, and is a generalisation of the congruence relations for $m$-ary partitions found by Andrews, Gupta, and R{\o}dseth and Sellers.
\end{abstract}

\maketitle

\section{Introduction}

Given a positive integer $n$, a partition of $n$ is an expression:
\begin{align*}
n=n_{1}+n_{2}+\cdots +n_{k},
\end{align*}
where all the numbers $n_{j}$ are integers and $1\leq n_{1}\leq n_{2}\leq \ldots \leq n_{k}$. Let $p(n)$ denote the number of partitions of a natural number $n$.

One of the most classical problems in the theory of partitions is studying the arithmetic properties of certain sequences of partitions. It originates in the celebrated work of Ramanujan, who proved that
\begin{align*}
    p(5n+4)\equiv &\ 0\pmod 5, \\
    p(7n+5)\equiv &\ 0\pmod 7,
\end{align*}    
and
\begin{align*}
    p(11n+6)\equiv &\ 0\pmod{11}
\end{align*}
hold for each nonnegative integer $n$. Similar results were proved also for other types of partitions, that is, for functions counting only the partitions of a given number that satisfy certain constrains.

In this paper we focus on the case of $m$-ary partitions. More precisely, these are partitions such that every part is a power of a fixed number $m$. For example, all the $2$-ary partitions of number $5$ are
\begin{align*}
1+1+1+1+1 = 1+1+1+2 = 1+2+2 = 1+4.
\end{align*}
In particular, the number of such partitions is equal to $4=:b_{2}(5)$. In general, there are no simple formulas for the number of $m$-ary partitions of $n$, and we will denote this number by $b_{m}(n)$.

In the context of $m$-ary partitions of various types, the problem of finding so called supercongruences, that is, congruences modulo high powers of $m$, is one of the most important parts of the theory. Churchhouse was the first to ask about arithmetic properties of the sequences of binary partitions. In \cite{Ch}, he proved that $b_{2}(n)$ is even if $n\geq 2$. He also characterized numbers $b_{2}(n)$ modulo $4$ and conjectured that for a fixed positive integer $k$ we have 
\begin{align*}
    b_{2}(2^{k+2}n)\equiv &\ b_{2}(2^{k}n)\pmod{2^{3k+2}}, \\
    b_{2}(2^{2k+1})\equiv &\ b_{2}(2^{2k-1})\pmod{2^{3k}}
\end{align*}
for all positive integers $n$. This conjecture was proved by R{\o}dseth in \cite{R}.

The above result was subsequently improved and extended to the case of $m$-ary partitions 
by Andrews \cite{Anda}, Gupta \cite{Gup}, and R{\o}dseth and Sellers \cite{RS}. Namely, they proved that
\begin{align}\label{eqbm}
    b_{m}(m^{r+1}n-\sigma -m)\equiv 0\ \ \ \left(\bmod{\frac{m^{r}}{c_{r}}}\right),
\end{align}
where for a sequence $(\varepsilon_{j})_{j=1}^{r-1}\in\{0,1\}^{r-1}$ we define
\begin{align*}
    \sigma :=\sum_{j=1}^{r-1}\varepsilon_{j}m^{j}, \ \ \ \ \ \ \ \ \ \ \  
    c_{r}:=\begin{cases}
        1 & \textrm{if $m$ is odd}, \\
        2^{r-1} & \textrm{if $m$ is even}.
    \end{cases}
\end{align*}

Similar results were obtained for many types of $m$-ary partitions including the case of $m$-ary partitions with no gaps \cite{ABRS, Hou} and $m$-ary overpartitions \cite{LuM}.

Folsom et al. introduced in \cite{Fol} the class of $M$-non-squashing partitions that generalizes ordinary $m$-ary partitions. Here, we call the $M$-non-squashing partitions the $M$-ary partitions for simplicity and due to their similarity with the ordinary $m$-ary partitions. More precisely, let $M=(m_{i})_{i=0}^{\infty}$ be a sequence of integers, $m_{0}=1$, $m_{i}\geq 2$ for $i\geq 1$, and for $r\geq 0$ let
\begin{align*}
M_{r}:=m_{0}m_{1}\cdots m_{r}.
\end{align*}
Then an $M$-ary partition of a positive integer $n$ is an expression of the form
\begin{align*}
n=M_{r_{1}}+M_{r_{2}}+\cdots +M_{r_{s}},
\end{align*}
for some nonnegative integers $r_{1},\ldots ,r_{s}$. We denote the number of $M$-ary partitions of $n$ by $p_{M}(n)$.

The natural question arises whether the congruences \eqref{eqbm} can be generalized to a more general class of $M$-partitions, and if not, whether we can prove some weaker congruences. In fact, our main motivation is the following conjecture stated in \cite{Fol}. 

\begin{conj}[\cite{Fol}, Cojecture 1.10]\label{HCongConj1}
Let sequences $M=(m_{j})_{j=0}^{\infty}$ and $\varepsilon= (\varepsilon_{j})_{j=1}^{r-1}\in\{0,1\}^{r-1}$ be fixed. Let
\begin{align} \label{HCongSigma}
\mu_{j}:=\frac{m_{j}}{(m_{j},P_{j})} \hspace{1cm} \textrm{ and } \hspace{1cm} \sigma=\sum_{j=1}^{r-1}\varepsilon_{j}\prod_{i=0}^{j}m_{i},
\end{align}
where $P_{j}=\prod_{p\leq j}p$ is the product of all primes up to $j$. Then for all integers $n\geq 1$ and $0\leq c\leq m_{1}-1$ we have that
\begin{align*}
p_{M}(m_{1}\cdots m_{r} n-\sigma -m_{1}-c)\equiv 0\pmod{\mu_{1}\cdots\mu_{r}}.
\end{align*}
\end{conj}

In the special case of partitions into factorials, the above conjecture can be stated as follows.

\begin{conj}[\cite{Fol}, Cojecture 1.11]\label{HCongConj2}
Let $M=(j)_{j=1}^{\infty}$. Then
\begin{align*}
p_{M}(r!n -\sigma -c)\equiv 0\pmod{r!/D_{r}},
\end{align*}
where $\sigma$ is the same as in \eqref{HCongSigma}, $c\in\{1,2\}$, and
\begin{align*}
D_{r}=\prod_{p\leq r-2}p^{\left\lfloor\frac{r-2}{p}\right\rfloor}.
\end{align*}
\end{conj}

Unfortunately, the above conjectures are false in general. Indeed, we checked, using a Mathematica \cite{Wol} that for example
\begin{align*}
p_{M}(5!n-2-4!)\equiv 10 \ \ \left({\rm mod}\ \frac{5!}{6}\right).
\end{align*}
for $n\in\{2,6,8,10,12,16\}$. In fact, Conjecture \ref{HCongConj2} failed in almost all cases we checked, that is, for sequences with $r\leq 5$. However, we made an attempt to find and prove some weaker results of the same type. We focused on the situation in which $\sigma =0$. Moreover, we prove the above conjectures only modulo a little worse modulus than the product $\mu_{1}\cdots\mu_{r}$.

For all positive integers $a$ and $b$ we define $\mu (a,b):=\frac{a}{(a,b)}$. Moreover, let
\begin{align*}
\mathcal{M}(m,r):=\gcd\{\ \mu(m,j)\ |\ j=1,\ldots ,r\ \}=\frac{m}{\gcd\big(m,\lcm (1,\ldots ,r)\big)}
\end{align*}
for positive integers $m$ and $r$.

The main result of the paper is the following.

\begin{theorem}\label{HCongMainThm}
Let $M=(m_{0},m_{1},\ldots )$ be a (possibly finite) sequence of integers such that $m_{0}=1$ and $m_{j}\geq 2$ if $j\geq 1$. Then for every positive integers $n$ and $r$ we have
\begin{align*}
p_{M}(m_{1}m_{2}\cdots m_{r}n-1)\equiv 0\ \ \left({\rm mod}\ \prod_{t=2}^{r}\mathcal{M}(m_{t},t-1)\right).
\end{align*} 
\end{theorem}

Let us provide some instances in which Conjecture \ref{HCongConj1} is true. In fact even stronger divisibility property holds if we assume that $m_{r}$ has only large prime factors for every $r$.

\begin{cor}
Assume for every prime $p$ that if $p\mid m_{r}$ then $p\geq r$. Then for every $n$ we have
\begin{align*}
p_{M}(m_{1}m_{2}\cdots m_{r}n-1)\equiv 0\ \ \left({\rm mod}\ \prod_{t=2}^{r}m_{r}\right).
\end{align*} 
\end{cor}
\begin{proof}
It is a direct consequence of Theorem \ref{HCongMainThm}.
\end{proof}

Finally, observe that 
\begin{align*}
\gcd \big(m,\lcm (1,2)\big)=\gcd (m,2)
\end{align*}
and
\begin{align*}
\gcd \big(m,\lcm (1,2,3)\big)=\gcd (m,2\cdot 3).
\end{align*}
Hence, we get that divisibility properties from Theorem \ref{HCongMainThm} are the same as in Conjectures \ref{HCongConj1} and \ref{HCongConj2} if $r\in\{1,2,3\}$ and $\sigma =0$.

\begin{cor}
Conjectures \ref{HCongConj1} and \ref{HCongConj2} are true if $r\in\{1,2,3\}$ and $\sigma =0$.
\end{cor}

\section*{Acknowledgements}
The author would like to express his gratitude to Piotr Miska and Maciej Ulas for careful reading of the manuscript and for many suggestions. He is also very grateful to the anonymous referees whose comments allowed to fix some mistakes that appeared in the previous version of the paper.

The author is supported by Czech Science Foundation, grant No. 21-00420M, and Charles University Research Centre program UNCE/SCI/022. This paper is a part of the author's PhD dissertation. During the preparation of the dissertation, the author was a scholarship holder of the Kartezjusz program funded by the Polish National Centre for Research and Development, grant No. POWR.03.02.00-00-I001/16-00.

\section{Auxiliary results}

In our proof, we will use the main idea from \cite{RS}. At first, we will connect the numbers of the form $p_{M}(m_{1}\ldots m_{r} n-1)$ with the numbers $\alpha_{m,r}(i)$ defined in the following lemma.

\begin{lem}\label{HCongLem0}
For any integers $m\geq 2$ and $n$, $r\geq 1$ there exist unique integers $\alpha_{m,r}(i)$ such that
\begin{align*}
\binom{mn+r-1}{r}=\sum_{i=1}^{r}\alpha_{m,r}(i)\binom{n+i-1}{i}.
\end{align*}
Moreover,
\begin{enumerate}
\item $\alpha_{m,r}(r)=m^{r}$,
\item $\alpha_{m,r}(r-1)=-\frac{1}{2}(r-1)(m-1)m^{r-1}$.
\end{enumerate}
\end{lem}
\begin{proof}
This is \cite[Lemma 1]{RS} together with remarks after the proof.
\end{proof}

In the proof, we will need the generating function of the sequence $(p_{M}(n))_{n=0}^{\infty}$. Let us denote it by $F_{M}(q)$:
\begin{align*}
F_{M}(q) = \sum_{n=0}^{\infty}p_{M}(n)q^{n} = \prod_{j=0}^{\infty}\frac{1}{1-q^{M_{j}}}.
\end{align*}
In particular, if $M'=(1,m_{2},m_{3},\ldots )$ then
\begin{align*}
F_{M}(q)=\frac{1}{1-q}\prod_{j=0}^{\infty}\frac{1}{1-q^{m_{1}M'_{j}}} = \frac{1}{1-q}F_{M'}(q^{m_{1}}).
\end{align*}

For a positive integer $m\geq 2$ let us define the following linear operator:
\begin{align*}
U_{m}:\mathbb{Z}\llbracket q\rrbracket\ni\sum_{n=0}^{\infty}a(n)q^{n}\longmapsto\sum_{n=0}^{\infty}a(mn)q^{n}\in \mathbb{Z}\llbracket q\rrbracket.
\end{align*}

The motivation behind considering the above operator is simple. The generating function of the sequence $\big(p_{M}(m_{1}\cdots m_{r} n-1)\big)_{n=1}^{\infty}$ is exactly $U_{m_{r}}\circ\cdots\circ U_{m_{1}}\big(qF_{M}(q)\big)$. We will explain this phenomenon in more details in the proof of Theorem \ref{HCongMainThm} below. In the proof, we will need some properties of the operators $U_{m}$ that we gather in the next lemma.

\begin{lem}\label{HCongLem1}
For every $m\geq 2$ the following properties hold.
\begin{enumerate}
\item If $f(q)=\sum_{n=0}^{\infty}a(n)q^{n}\in\mathbb{Z}\llbracket q\rrbracket$, then
\begin{align*}
U_{m}\big(qf(q)\big)=\sum_{n=0}^{\infty}a(mn-1)q^{n}.
\end{align*}
\item If $f(q)$, $g(q)\in\mathbb{Z}\llbracket q\rrbracket$, then
\begin{align*}
U_{m}\big(f(q)g(q^{m})\big)=U_{m}(f(q))g(q).
\end{align*}
\end{enumerate}
\end{lem}
\begin{proof}
This is a direct consequence of the definition of the operator $U_{m}$.
\end{proof}

As we explained earlier, the generating function of the sequence $\big(p_{M}(m_{1}\cdots m_{r} n-1)\big)_{n=1}^{\infty}$ is an image of the series $qF_{M}(q)$ under the operator $U_{m_{r}}\circ\cdots\circ U_{m_{1}}$. On the other hand, it will be possible to express the same sequence in a more manageable form. In order to do so, we introduce the following sequences of power series $h_{r}$ and $H_{(m_{1},\ldots ,m_{r})}$. Let
\begin{align*}
h_{r}=h_{r}(q):=\frac{q}{(1-q)^{r+1}}
\end{align*}
for $r\geq 0$.

\begin{lem}\label{HCongLem2}
For each $r$ and $m$ we have
\begin{align*}
U_{m}h_{r}=\sum_{i=1}^{r}\alpha_{m,r}(i)h_{i},
\end{align*}
where numbers $\alpha_{m,r}(i)$ are the same as in Lemma \ref{HCongLem0}.
\end{lem}
\begin{proof}
Follows immediately from the definition of $U_{m}$ and Lemma \ref{HCongLem1}.
\end{proof}

For a sequence $(m_{1},m_{2},\ldots )$ of natural numbers greater than or equal to $2$, let us define the following sequence of power series with integer coefficients (all these sequences depend on $q$ so we skip the variable for simplicity of notation):
\begin{align*}
H_{\emptyset} := &\ h_{0}=\frac{q}{1-q}, \\
H_{(m_{1})} := &\ U_{m_{1}}\left(\frac{1}{1-q}H_{\emptyset}\right),
\end{align*}
and if $H_{(m_{1},\ldots ,m_{r-1})}$ is defined for some $r\geq 2$ then
\begin{align*}
H_{(m_{1},\ldots ,m_{r-1},m_{r})}:= U_{m_{r}}\left(\frac{1}{1-q}H_{(m_{1},\ldots ,m_{r-1})}\right).
\end{align*}

\begin{exam}\label{HCongExample1}
For $r\leq 3$ the series defined above satisfy the following equalities:
\begin{enumerate}
\item $H_{(m_{1})}=m_{1}h_{1}$,
\item $H_{(m_{1},m_{2})}=m_{1}m_{2}^{2}h_{2}-\binom{m_{2}}{2}H_{(m_{1})}$,
\item $H_{(m_{1},m_{2},m_{3})}=m_{1}m_{2}^{2}m_{3}^{3}h_{3}-m_{3}^{2}(m_{3}-1)H_{(m_{1},m_{2})}-\binom{m_{2}}{2}H_{(m_{1},m_{3})}$ \\
 \hspace*{8.7cm} $-\left(m_{3}^{2}(m_{3}-1)\binom{m_{2}}{2}-m_{2}^{2}\binom{m_{3}}{3}\right)H_{(m_{1})}$.
\end{enumerate}
\end{exam}
Note here that the above representations of $H_{(m_1,\ldots,m_r)}$ as linear combinations of $H_{(m_{j_1},\ldots,m_{j_s})}$, $s<r$, $j_1<\ldots<j_s$, may NOT be unique. However, in the sequel we will need only the existence of such representations which is presented below.

\begin{lem}\label{HCongLem3}
Let $r\geq 1$ and $m_{j}\geq 2$ for $j\in\{1,\ldots ,r\}$. Then there exist (non-unique) integers $\beta_{(j_{1},\ldots ,j_{s})}(m_{1},\ldots ,m_{r})$ such that
\begin{align*}
H_{(m_{1},\ldots ,m_{r})}=m_{1}m_{2}^{2}\cdots m_{r}^{r}h_{r}-\sum_{s=1}^{r-1}\ \sum_{1\leq j_{1}<\ldots <j_{s}\leq r}\beta_{(j_{1},\ldots ,j_{s})}(m_{1},\ldots ,m_{r})H_{(m_{j_{1}},\ldots ,m_{j_{s}})}.
\end{align*}
Moreover, we can chose $\beta_{(j_{1},\ldots ,j_{s})}(m_{1},\ldots ,m_{r})$ so that they satisfy the following recurrence relations: 
\begin{itemize}
\item if $j_{s}=r$, then if $s\geq 2$,
\begin{align*}
\beta_{(j_{1},\ldots ,j_{s-1},r)}(m_{1},\ldots ,m_{r})=\beta_{(j_{1},\ldots ,j_{s-1})}(m_{1},\ldots ,m_{r-1}),
\end{align*}
and if $s=1$,
\begin{align*}
\beta_{(r)}(m_{1},\ldots ,m_{r})=0.
\end{align*}
\item if $j_{s}\neq r$ and $(j_{1},\ldots ,j_{s})=(r-s,r-s+1,\ldots ,r-1)$, then
\begin{align*}
\beta_{(r-s,\ldots ,r-1)}(m_{1},\ldots  & ,m_{r})=\left(\prod_{j=1}^{r-s-1}m_{j}^{j}\prod_{j=r-s}^{r-1}m_{j}^{r-s-1}\right)\alpha_{m_{r},r}(s) \\ & +\sum_{i=s+1}^{r-1}\left(\prod_{j=1}^{r-i-1}m_{j}^{j}\prod_{j=r-i}^{r-1}m_{j}^{r-i-1}\right)\alpha_{m_{r},r}(i)\beta_{(i-s+1,\ldots ,i)}(m_{r-i},\ldots ,m_{r-1}).
\end{align*}
\item otherwise,
\begin{align*}
\beta_{(j_{1},\ldots ,j_{s})} & (m_{1},\ldots ,m_{r}) \\ 
& = \sum_{i=s+1}^{r-1}\left(\prod_{j=1}^{r-i-1}m_{j}^{j}\prod_{j=r-i}^{r-1}m_{j}^{r-i-1}\right)\alpha_{m_{r},r}(i) \beta_{(j_{1}-(r-i)+1,\ldots ,j_{s}-(r-i)+1)}(m_{r-i},\ldots ,m_{r-1}).
\end{align*}
\end{itemize}
\end{lem}
\begin{proof}
We use induction on $r$. For $r=1$ and $r=2$ the result is presented in Example \ref{HCongExample1}. Let us assume that $r\geq 3$ and for every $1\leq k\leq r-1$ and every sequence $(n_{1},\ldots ,n_{k})$, where $n_{j}\geq 2$ for each $j\in\{1,\ldots ,k\}$, we have
\begin{align}\label{HConIndHyp}
H_{(n_{1},\ldots ,n_{k})}=n_{1}n_{2}^{2}\cdots n_{k}^{k}h_{k}-\sum_{s=1}^{k-1}\sum_{1\leq j_{1}<\ldots <j_{s}\leq k}\beta_{(j_{1},\ldots ,j_{s})}(n_{1},\ldots ,n_{k})H_{(n_{j_{1}},\ldots ,n_{j_{s}})}.
\end{align}

From the recurrence relations defining $H_{(m_{1},\ldots ,m_{r})}$, induction hypothesis \eqref{HConIndHyp} used for $k=r-1$ and $(n_{1},\ldots ,n_{k})=(m_{1},\ldots ,m_{r-1})$, and linearity of the operator $U_{m_{r}}$, we get
\begin{align*}
H_{(m_{1},\ldots ,m_{r})}= &\ U_{m_{r}}\left(\frac{1}{1-q}H_{(m_{1},\ldots ,m_{r-1})}\right)
= m_{1}\cdots m_{r-1}^{r-1}U_{m_{r}}\left(\frac{1}{1-q} h_{r-1}\right) \\
 & -\sum_{s=1}^{r-2}\sum_{1\leq j_{1}<\ldots <j_{s}\leq r-1}\beta_{(j_{1},\ldots ,j_{s})}(m_{1},\ldots ,m_{r-1})U_{m_{r}}\left(\frac{1}{1-q} H_{(m_{j_{1}},\ldots ,m_{j_{s}})}\right) \\
 = &\ m_{1}\cdots m_{r-1}^{r-1}U_{m_{r}}\big(h_{r}\big)-\sum_{s=1}^{r-2}\sum_{1\leq j_{1}<\ldots <j_{s}\leq r-1}\beta_{(j_{1},\ldots ,j_{s})}(m_{1},\ldots ,m_{r-1})H_{(m_{j_{1}},\ldots ,m_{j_{s}},m_{r})}.
\end{align*}
In the remaining part of the proof we will not modify the coefficients of $H_{(m_{j_{1}},\ldots ,m_{j_{s}},m_{r})}$. Hence, we can take 
\begin{align*}
\beta_{(j_{1},\ldots ,j_{s},r)}(m_{1},\ldots ,m_{r})=\beta_{(j_{1},\ldots ,j_{s})}(m_{1},\ldots ,m_{r-1}).
\end{align*}
Moreover, in the above sum the expression $H_{(m_{r})}$ does not appear. Therefore,
\begin{align*}
\beta_{(r)}(m_{1},\ldots ,m_{r})=0,
\end{align*}
as in the statement.

Lemmas \ref{HCongLem2} and \ref{HCongLem0} imply
\begin{align*}
 m_{1}\cdots m_{r-1}^{r-1}U_{m_{r}}\big(h_{r}\big)= &\ m_{1}\cdots m_{r-1}^{r-1}\left(\sum_{i=1}^{r}\alpha_{m_{r},r}(i)h_{i}\right) \\
 = &\ m_{1}\cdots m_{r-1}^{r-1}\left(m_{r}^{r}h_{r}+\sum_{i=1}^{r-1}\alpha_{m_{r},r}(i)h_{i}\right) \\
 = &\ m_{1}\cdots m_{r-1}^{r-1}m_{r}^{r}h_{r}+m_{1}\cdots m_{r-1}^{r-1}\sum_{i=1}^{r-1}\alpha_{m_{r},r}(i)h_{i}.
\end{align*}

Observe, that the induction hypothesis \eqref{HConIndHyp} for every $k=i\leq r-1$ and every sequence $(n_{1},\ldots ,n_{k})=(m_{l_{1}},\ldots ,m_{l_{i}})$, where $1\leq l_{1}<\ldots <l_{i}\leq r-1$, implies
\begin{align*}
m_{l_{1}}m_{l_{2}}^{2}\cdots m_{l_{i}}^{i}h_{i}=H_{(m_{l_{1}},\ldots ,m_{l_{i}})}+\sum_{s=1}^{i-1}\ \sum_{1\leq j_{1}<\ldots <j_{s}\leq i}\beta_{(j_{1},\ldots ,j_{s})}(m_{l_{1}},\ldots ,m_{l_{i}})H_{(m_{l_{j_{1}}},\ldots ,m_{l_{j_{s}}})}.
\end{align*}
Let $1\leq i\leq r-1$ be fixed and apply the above equality with $(l_{1},\ldots ,l_{i})=(r-i,r-i+1,\ldots ,r-1)$ to get
\begin{align*}
\left(\prod_{j=1}^{i}m_{r-1-i+j}^{j}\right)h_{i} & =H_{(m_{r-i},\ldots ,m_{r-1})} \\ & +\sum_{s=0}^{i-1}\ \sum_{r-i\leq j_{1}<\ldots <j_{s}\leq r-1}\beta_{(j_{1}-(r-i)+1,\ldots ,j_{s}-(r-i)+1)}(m_{r-i},\ldots ,m_{r-1})H_{(m_{j_{1}},\ldots ,m_{j_{s}})}.
\end{align*}
The above relation implies
\begin{align*}
m_{1} &  \cdots m_{r-1}^{r-1}\sum_{i=1}^{r-1}\alpha_{m_{r},r}(i)h_{i}=  \sum_{i=1}^{r-1}\left(\prod_{j=1}^{r-i-1}m_{j}^{j}\prod_{j=r-i}^{r-1}m_{j}^{r-i-1}\right)\alpha_{m_{r},r}(i)\left(\prod_{j=1}^{i}m_{r-1-i+j}^{j}\right)h_{i} \\
= &\ \sum_{i=1}^{r-1}\left(\prod_{j=1}^{r-i-1}m_{j}^{j}\prod_{j=r-i}^{r-1}m_{j}^{r-i-1}\right)\alpha_{m_{r},r}(i)H_{(m_{r-i},\ldots ,m_{r-1})}  +\sum_{i=1}^{r-1}\left(\prod_{j=1}^{r-i-1}m_{j}^{j}\prod_{j=r-i}^{r-1}m_{j}^{r-i-1}\right)\alpha_{m_{r},r}(i)\times \\ & \times \sum_{s=0}^{i-1}\ \sum_{r-i\leq j_{1}<\ldots <j_{s}\leq r-1}\beta_{(j_{1}-(r-i)+1,\ldots ,j_{s}-(r-i)+1)}(m_{r-i},\ldots ,m_{r-1})H_{(m_{j_{1}},\ldots ,m_{j_{s}})}.
\end{align*}

Now we can find the remaining expressions for $\beta_{(j_{1},\ldots ,j_{s})}(m_{1},\ldots ,m_{r})$. If $(j_{1},\ldots ,j_{s})=(r-s,r-s+1,\ldots ,r-1)$, then
\begin{align*}
\beta & _{(r-s,\ldots ,r-1)}(m_{1},\ldots ,m_{r})= \left(\prod_{j=1}^{r-s-1}m_{j}^{j}\prod_{j=r-s}^{r-1}m_{j}^{r-s-1}\right)\alpha_{m_{r},r}(s) \\ &\ \ \ \ \ +\sum_{i=s+1}^{r-1}\left(\prod_{j=1}^{r-i-1}m_{j}^{j}\prod_{j=r-i}^{r-1}m_{j}^{r-i-1}\right)\alpha_{m_{r},r}(i)\beta_{((r-s)-(r-i)+1,\ldots ,(r-1)-(r-i)+1)}(m_{r-i},\ldots ,m_{r-1}) \\
 &\ = \left(\prod_{j=1}^{r-s-1}m_{j}^{j}\prod_{j=r-s}^{r-1}m_{j}^{r-s-1}\right)\alpha_{m_{r},r}(s) \\ &\ \ \ \ \ +\sum_{i=s+1}^{r-1}\left(\prod_{j=1}^{r-i-1}m_{j}^{j}\prod_{j=r-i}^{r-1}m_{j}^{r-i-1}\right)\alpha_{m_{r},r}(i)\beta_{(i-s+1,\ldots ,i)}(m_{r-i},\ldots ,m_{r-1}).
\end{align*}
Otherwise we have
\begin{align*}
\beta_{(j_{1},\ldots ,j_{s})}(m_{1}, & \ldots ,m_{r}) \\ 
& = \sum_{i=s+1}^{r-1}\left(\prod_{j=1}^{r-i-1}m_{j}^{j}\prod_{j=r-i}^{r-1}m_{j}^{r-i-1}\right)\alpha_{m_{r},r}(i) \beta_{(j_{1}-(r-i)+1,\ldots ,j_{s}-(r-i)+1)}(m_{r-i},\ldots ,m_{r-1}) 
\end{align*}
(where we assumed that $\beta_{(k_{1},\ldots ,k_{u})}(n_{1},\ldots ,n_{u})=0$ if for some $i\in\{1,\ldots ,u\}$ we have $k_{i}\leq 0$). These are formulas we wanted to prove.
\end{proof}

Recall that for every positive integers $a$, $b$ we denote $\mu(a,b):=\frac{a}{(a,b)}$. 

\begin{lem}\label{HCongDivAlpha}
For all integers $m\geq 2$, $r\geq 2$, $i\in\{1,\ldots ,r\}$ we have
\begin{align*}
\mu(mi,r)\mid \alpha_{m,r}(i).
\end{align*}
In particular, $\mu(m,r)\mid \alpha_{m,r}(i)$ for all $i$.
\end{lem}
\begin{proof}
The number $m$ will not change during the proof so we omit it and write $\alpha_{r}(i)$ for $\alpha_{m,r}(i)$ in order to simplify the notation. For every $n$ we have:
\begin{align*}
\sum_{i=1}^{r}\alpha_{r}(i) & \binom{n+i-1}{i}=\binom{mn+r-1}{r}=\frac{mn+r-1}{r}\binom{mn+r-2}{r-1} \\ 
= &\ \frac{mn+r-1}{r}\sum_{i=1}^{r-1}\alpha_{r-1}(i)\binom{n+i-1}{i} \\
= &\ \frac{1}{r}\sum_{i=1}^{r-1}\alpha_{r-1}(i)\big(m(n+i)-(mi-r+1)\big)\binom{n+i-1}{i} \\
= &\ \frac{1}{r}\sum_{i=1}^{r-1}\alpha_{r-1}(i)\left[m(i+1)\binom{n+i}{i+1}-(mi-r+1)\binom{n+i-1}{i}\right] \\
= &\ \sum_{i=2}^{r}\frac{mi}{r}\alpha_{r-1}(i-1)\binom{n+i-1}{i}-\sum_{i=1}^{r-1}\frac{mi-r+1}{r}\alpha_{r-1}(i)\binom{n+i-1}{i} \\
= &\ m\alpha_{r-1}(r-1)\binom{n+r-1}{r} \\ 
 &\ +\sum_{i=2}^{r-1}\left[\frac{mi}{r}\alpha_{r-1}(i-1)-\frac{mi-r+1}{r}\alpha_{r-1}(i)\right]\binom{n+i-1}{i} -\frac{m-r+1}{r}\alpha_{r-1}(1)\binom{n}{1}.
\end{align*}
By comparing the coefficients of $\binom{n+i-1}{i}$ we get
\begin{align*}
\alpha_{r}(i)=\frac{mi}{r}\alpha_{r-1}(i-1)-\frac{mi-r+1}{r}\alpha_{r-1}(i),
\end{align*}
or equivalently,
\begin{align*}
r\alpha_{r}(i)-(r-1)\alpha_{r-1}(i)=mi\big(\alpha_{r-1}(i-1)-\alpha_{r-1}(i)\big).
\end{align*}
We can write down the same equalities with $r$ replaced by any $s\leq r$ to get
\begin{align*}
s\alpha_{s}(i)-(s-1)\alpha_{s-1}(i)=mi\big(\alpha_{s-1}(i-1)-\alpha_{s-1}(i)\big).
\end{align*}
Let us sum all these equalities by sides over $1\leq s\leq r$. We get
\begin{align*}
r\alpha_{r}(i)=mi\sum_{s=1}^{r}\big(\alpha_{s-1}(i-1)-\alpha_{s-1}(i)\big).
\end{align*}
The main part of the statement follows. The second part is a consequence of the fact that $\mu (m,r)\mid \mu(mi,r)$ for every $i$.
\end{proof}

\begin{lem}\label{HCongLemDivBeta}
For every $r\geq 1$ and every $1\leq j_{1}<\ldots <j_{s}\leq r$ we have
\begin{align*}
\beta_{(j_{1},\ldots ,j_{s})}(m_{1},\ldots ,m_{r})\equiv 0 \ \ \bigg({\rm mod} \prod_{\substack{1\leq t\leq r \\ t\not\in\{j_{1},\ldots ,j_{s}\}}}\mu (m_{t},t)\bigg).
\end{align*}
If $m_{1}=m_{2}=\ldots =m$, then $\beta_{(j_{1},\ldots ,j_{s})}(m_{1},\ldots ,m_{r})\equiv 0\pmod{m^{r-s}/(m,2)}$.
\end{lem}
\begin{proof}
We proceed by induction on $r$. The cases $r=1$, $r=2$ and $r=3$ simply follow from Example \ref{HCongExample1}. Let us assume that the statement is true for $r-1$ and consider recurrence relations from Lemma \ref{HCongLem3} modulo the product from the statement. We get the following cases:
\begin{itemize}
\item if $j_{s}=r$, then by the first case of Lemma \ref{HCongLem3}:
\begin{align*}
\beta_{(j_{1},\ldots ,j_{s-1},r)}(m_{1},\ldots ,m_{r})=\beta_{(j_{1},\ldots ,j_{s-1})}(m_{1},\ldots ,m_{r-1}),
\end{align*}
so the statement follows from the induction hypothesis,
\item if $(j_{1},\ldots ,j_{s})=(r-s,\ldots ,r-1)$, then for $s\leq r-2$ all the summands on the right hand side in the formula from the second case of Lemma \ref{HCongLem3} are divisible by $m_{1}m_{2}\cdots m_{r-1}\cdot \alpha_{m_{r},r}(i)$ for some $i$. Thus the result follows from Lemma \ref{HCongDivAlpha}. If $s=r-1$ then $(j_{1},\ldots ,j_{s})=(1,2,\ldots ,r-1)$, so the result again follows since every summand is divisible by $\alpha_{m_{r},r}(i)$ for some $i$,
\item the remaining case: similarly as in the previous case, every summand in the formula from the last case of Lemma \ref{HCongLem3} is divisible by $m_{1}\cdots m_{r-1}\cdot \alpha_{m_{r},r}(i)$ for some $i$. Hence, the result follows from Lemma \ref{HCongDivAlpha}.
\end{itemize}
The first part of the statement follows. For the second one let us denote
\begin{align*}
N:=\prod_{j=1}^{r}m_{j},
\end{align*}
Lemma \ref{HCongLem3} implies that for every sequence $(m_{1},m_{2},\ldots)$ the following congruences hold:
\begin{itemize}
\item $\beta_{(j_{1},\ldots ,j_{s-1},r)}(m_{1},\ldots ,m_{r})\equiv \beta_{(j_{1},\ldots ,j_{s-1})}(m_{1},\ldots ,m_{r-1}) \pmod{N}$,
\item if $j_{1}=r-s$ and $s\neq r-1$ then
\begin{align*}
\beta_{(r-s,\ldots ,r-1)}(m_{1},\ldots ,m_{r})\equiv \alpha_{m_{r},r}(r-1)\beta_{(r-s,\ldots ,r-1)}(m_{1},\ldots ,m_{r-1}) \pmod{N},
\end{align*}
and if $s=r-1$ then
\begin{align*}
\beta_{(1,\ldots ,r-1)}(m_{1},\ldots ,m_{r})\equiv \alpha_{m_{r},r}(r-1)\pmod{N},
\end{align*}
\item if $j_{1}\neq r-s$ and $j_{s}\neq r$ then
\begin{align*}
\beta_{(j_{1},\ldots ,j_{s})}(m_{1},\ldots ,m_{r})\equiv \alpha_{m_{r},r}(r-1)\beta_{(j_{1},\ldots ,j_{s})}(m_{1},\ldots ,m_{r-1}) \pmod{N}.
\end{align*}
\end{itemize}

Let assume that all the numbers $m_{j}$ are equal to some number $m$. Then the result follows quickly in the second and third case from the fact that $\alpha_{m,r}(r-1)=-\frac{1}{2}(r-1)(m-1)m^{r-1}$ is divisible by $m^{r-1}/(m,2)$.

Let us now move to the first case. Observe that if $t<s$ is such that $(j_{1},\ldots ,j_{s})=(j_{1},\ldots ,j_{s-t},r-t+1,r-t+2,\ldots ,r)$ and $j_{s-t}\neq r-t$, then
\begin{align*}
\beta_{(j_{1},\ldots ,j_{s})}(\underbrace{m,\ldots ,m}_{r \textrm{ times}}) \equiv \beta_{(j_{1},\ldots ,j_{s-1})}(\underbrace{m,\ldots ,m}_{r-1 \textrm{ times}}) \equiv\ldots\equiv \beta_{(j_{1},\ldots ,j_{s-t})}(\underbrace{m,\ldots ,m}_{r-t \textrm{ times}}) \pmod{m^{r-t}}.
\end{align*}
By the previously considered cases, the last quantity is divisible by
\begin{align*}
\frac{m^{(r-t)-(s-t)}}{(m,2)}=\frac{m^{r-s}}{(m,2)},
\end{align*} 
and so is $\beta_{(j_{1},\ldots ,j_{s})}(\underbrace{m,\ldots ,m}_{r \textrm{ times}})$ in this case, as we wanted.

The last case that we need to consider is $(j_{1},\ldots ,j_{s})=(r-s+1,\ldots ,r)$. Then 
\begin{align*}
\beta_{(r-s+1,\ldots ,r)}(\underbrace{m,\ldots ,m}_{r\textrm{ times}}) =\beta_{(r-s+1,\ldots ,r-1)}(\underbrace{m,\ldots ,m}_{r-1\textrm{ times}}) =\ldots = \beta_{(r-s+1)}(\underbrace{m,\ldots ,m}_{r-s+1\textrm{ times}}) =0
\end{align*}
by the subcase $s=1$ of the first case of of Lemma \ref{HCongLem3}. The result follows.
\end{proof}

For positive integers $m$ and $r$ let 
\begin{align*}
\mathcal{M}(m,r):=\gcd\{\ \mu(m,j)\ |\ j=1,\ldots ,r\ \}=\frac{m}{\gcd\big(m,\lcm (1,\ldots ,r)\big)}.
\end{align*}

\begin{lem}\label{HCongMainDivLem}
For every $n\in\mathbb{N}_{0}$ the coefficient of $q^{n}$ in the series $H_{(m_{1},\ldots ,m_{r})}$ is divisible by the product $\prod_{t=1}^{r}\mathcal{M}(m_{t},t)$.
\end{lem}
\begin{proof}
We use induction on $r$. For $r=1$ and any $m_{1}$ we have $H_{(m_{1})}=m_{1}h_{1}$ so indeed, all the coefficients of $H_{(m_{1})}$ are divisible by $m_{1}$. Assume that if $s\leq r-1$, then for every sequence $(n_{1},\ldots ,n_{s})$ with $n_{i}\geq 2$ the series $H_{(n_{1},\ldots ,n_{s})}$ is divisible by $\prod_{t=1}^{s}\mathcal{M}(n_{t},t)$. The latter product is divisible by $\prod_{t=1}^{s}\mathcal{M}(m_{j_{t}},j_{t})$, where $n_{t}=m_{j_{t}}$ for each $t$. The result follows from Lemmas \ref{HCongLem3} and \ref{HCongLemDivBeta}.
\end{proof}

\section{Proof of Theorem \ref{HCongMainThm}}

Because of Lemma \ref{HCongMainDivLem} it is enough to prove the following relation:
\begin{align}\label{HCongEquHF}
\sum_{n=1}^{\infty}p_{M}(m_{1}m_{2}\cdots m_{r}n-1)q^{n}=H_{(m_{2},m_{3},\ldots ,m_{r})}(q)F_{r}(q),
\end{align}
for $r\geq 2$, where $F_{r}(q):=F_{(m_{r+1},m_{r+2},\ldots )}(q)$. Let us begin with the case of $r=2$ (note that we have a string $(m_{2},\ldots ,m_{r})$ on the right hand side of the above inequality so Lemma \ref{HCongMainDivLem} has to be used with this string instead of $(m_{1},\ldots ,m_{r})$). We apply the operators $U_{m_{1}}$ and $U_{m_{2}}$ to both sides of the equality $qF_{0}(q)=\frac{q}{1-q}F_{1}(q^{m_{1}})$ (note that $F_{0}(q)=F_{M}(q)$ and $F_{1}(q)=F_{M'}(q)$). We get
\begin{align*}
\sum_{n=1}^{\infty}p_{M}(m_{1}m_{2}n-1)q^{n}= &\ U_{m_{2}}\circ U_{m_{1}}\left(qF_{0}(q)\right)=U_{m_{2}}\left( U_{m_{1}}\left(\frac{q}{1-q}F_{1}(q^{m_{1}})\right)\right) \\
= &\ U_{m_{2}}\left( U_{m_{1}}\left(\frac{q}{1-q}\right)F_{1}(q) \right)=U_{m_{2}}\left( U_{m_{1}}\left(\frac{q}{1-q}\right)\frac{1}{1-q}
F_{2}(q^{m_{2}}) \right) \\
= &\ U_{m_{2}}\left(\frac{q}{(1-q)^{2}}F_{2}(q^{m_{2}})\right)=U_{m_{2}}\left(\frac{q}{(1-q)^{2}}\right) F_{2}(q) = H_{(m_{2})}(q)F_{2}(q).
\end{align*}
In the above chain of equalities we used the fact that $U_{m}\left(\frac{q}{1-q}\right)=\frac{q}{1-q}$ for every $m\geq 2$.

Let us now assume that \eqref{HCongEquHF} is true for some $r$. We want to prove it for $r+1$. Let us apply the operator $U_{m_{r+1}}$ to both sides of \eqref{HCongEquHF} and get:
\begin{align*}
\sum_{n=1}^{\infty}p_{M}(m_{1}m_{2}\cdots m_{r}m_{r+1}n-1)q^{n}= &\ U_{m_{r+1}}\left(H_{(m_{2},\ldots ,m_{r})}(q)F_{r}(q)\right) \\
= &\ U_{m_{r+1}}\left(H_{(m_{2},\ldots ,m_{r})}(q)\frac{1}{1-q}F_{r+1}\left(q^{m_{r+1}}\right)\right) \\
= &\ U_{m_{r+1}}\left(\frac{1}{1-q}H_{(m_{2},\ldots ,m_{r})}(q)\right)F_{r+1}\left(q\right) \\
= &\ H_{(m_{2},\ldots ,m_{r},m_{r+1})}(q)F_{r+1}\left(q\right) .
\end{align*}
The proof is finished. \hfill $\square$

\begin{rem}
One may ask what happens if we use the idea from the above proof to find an expression for the generating function of $(p_{M}(m_{1}n-1))_{n=1}^{\infty}$. We would have
\begin{align*}
\sum_{n=1}^{\infty}p_{M}(m_{1}n-1)q^{n} & = U_{m_{1}}(qF_{M}(q))=U_{m_{1}}\left(\frac{q}{1-q} F_{M'}(q^{m_{1}})\right) \\
& =U_{m_{1}}\left(\frac{q}{1-q}\right)F_{M'}(q) = \frac{q}{1-q}F_{M'}(q).
\end{align*}
We can now use equality $p_{M}(m_{1}n-1)=p_{M}(m_{1}(n-1))$ that comes quickly by comparing the coefficients in the equality $(1-q)F_{M}(q)=F_{M'}(q^{m_{1}})$. We get
\begin{align}\label{EquRandom}
(1-q)\sum_{n=1}^{\infty}p_{M}(m_{1}(n-1))q^{n} = qF_{M'}(q).
\end{align}
After expanding the power series and comparing the coefficients we get that equality \eqref{EquRandom} is equivalent to
\begin{align*}
p_{M}(m_{1}n)-p_{M}(m_{1}(n-1)) = p_{M'}(n)
\end{align*}
for every $n\geq 1$. However, the same equality follows quickly by comparing the coefficients in the equation $(1-q)F_{M}(q)=F_{M'}(q^{m_{1}})$ so we achieved a relation that is true but does not give any new information.
\end{rem}

\end{document}